\documentclass[11pt]{amsart}
\usepackage{amssymb,pstricks,pst-plot}
\parskip=\smallskipamount

\newcommand{\B}{\mathbb{B}}
\newcommand{\C}{\mathbb{C}}
\newcommand{\D}{\mathbb{D}}
\newcommand{\clD}{{\overline{\,\D}}}

\newcommand{\N}{\mathbb{N}}

\renewcommand{\P}{\mathbb{P}}
\newcommand{\R}{\mathbb{R}}

\newcommand{\bT}{\mathbb{T}}

\renewcommand{\d}{\mathrm{d}}

\newcommand\E{\mathrm{e}}
\newcommand\I{\mathrm{i}}

\newcommand{\cA}{\mathcal{A}}
\newcommand{\cB}{\mathcal{B}}

\newcommand{\cF}{\mathcal{F}}

\newcommand{\cL}{\mathcal{L}}

\newcommand{\cO}{\mathcal{O}}

\newcommand\wt{\widetilde}

\newcommand{\Psh}{\mathrm{Psh}}

\newcommand{\ol}{\overline}

\newcommand\wh{\widehat}
\newcommand\reg{\mathrm{reg}}
\newcommand\sing{\mathrm{sing}}
\newcommand\bs{\backslash}

\newtheorem{thm}{Theorem}[section]

\newtheorem{lem}[thm]{Lemma}
\newtheorem{prob}[thm]{Problem}

\theoremstyle{definition}
\newtheorem{defin}[thm]{Definition}
\newtheorem{rem}[thm]{Remark}
\newtheorem{exa}[thm]{Example}

\begin{document}





\title[Disc functionals and Siciak-Zaharyuta extremal functions]
{Disc functionals and Siciak-Zaharyuta extremal functions on singular varieties} 

\author[B. Drinovec Drnov\v sek]{Barbara Drinovec Drnov\v sek}
\address{Faculty of Mathematics and Physics\\ University of Ljubljana\\
Institute of Mathematics, Physics and Mechanics\\
Jadranska 19, 1000 Ljubljana, Slovenia}
\email{barbara.drinovec@fmf.uni-lj.si}

\author[F. Forstneri\v c]{Franc Forstneri\v c}
\address{Faculty of Mathematics and Physics\\ University of Ljubljana\\
Institute of Mathematics, Physics and Mechanics\\
Jadranska 19, 1000 Ljubljana, Slovenia}
\email{franc.forstneric@fmf.uni-lj.si}

\date{\today}

\begin{abstract}
We establish plurisubharmonicity of envelopes 
of certain classical disc functionals on locally irreducible 
complex spaces, thereby generalizing the corresponding results for complex manifolds.
We also find new formulae expressing the Siciak-Zaharyuta extremal function
of an open set in a locally irreducible affine algebraic variety 
as the envelope of certain disc functionals, 
similarly to what has been done for open sets in $\C^n$ 
by Lempert and by L\'arusson and Sigurdsson. 
\end{abstract}



\maketitle

{\small \rightline{\em  Dedicated to J\'ozef Siciak on the occasion of his 80th birthday}}

\section{Introduction}
\label{sec:introduction}
Let $\D=\{\zeta\in \C\colon |\zeta|< 1\}$ be the open unit disc in 
the complex plane $\C$, and let $\bT=b\D=\{\zeta\in\C\colon |\zeta|=1\}$ 
be its boundary circle.

Let $X$ be a (reduced, paracompact) complex space.
Denote by $\Psh(X)$ the set of all plurisubharmonic functions on $X$.
(For convenience we agree that the function which is identically 
equal to $-\infty$ is also plurisubharmonic.) 
Let $\cO(\clD,X)$ be the set of all maps $f\colon \clD \to X$ 
that are holomorphic in an open neighborhood $U_f\subset \C$ 
of the closed disc $\clD$ in $\C$. Given a point $x\in X$, set 
$\cO(\clD,X,x)=\{f\in \cO(\clD,X)\colon f(0)=x\}$;
these are holomorphic discs in $X$ centered at $x$.
A {\em disc functional} on $X$ is a function 
\[
	H_X \colon \cO(\clD,X) \to\ol\R= [-\infty,+\infty]. 
\]
The {\em envelope} of $H_X$ is the function $EH_X\colon X\to\ol\R$
defined by 
\begin{equation}
\label{eq:envelope}
		EH_X(x)= \inf \left\{ H_X(f)\colon f\in \cO(\clD,X,x),\    
		f(\D)\not\subset X_{\rm sing}	\right\}, \quad x\in X.	
\end{equation}
Occasionally we consider disc functionals as functions on the larger class 
$\cA_X=\cA(\D,X)$ of discs $\ol\D \to X$ that are holomorphic in 
$\D$ and continuous on $\clD$. For all functionals treated in this paper, 
their envelope over $\cA_X$ coincides with 
the envelope (\ref{eq:envelope}) over the subclass $\cO(\clD,X)$ of $\cA_X$.

The theory of disc functionals, initiated by 
Poletsky in the late 1980s \cite{Poletsky1991}, 
offers a different approach to certain 
extremal functions of pluripotential theory.
(For the latter subject see Klimek \cite{Klimek}.)
In several natural examples, the envelope of a disc functional is 
a plurisubharmonic function. 
Furthermore, extremal plurisubharmonic functions are usually 
defined as suprema of classes of plurisubharmonic functions 
with certain properties, and many of them are 
envelopes of appropriate disc functionals. 
As was pointed out by Poletsky in \cite{Poletsky2002},
one may view this subject as an extension of 
Kiselman's minimum principle \cite{Kiselman1978}.

In this paper we extend results on plurisubharmonicity
of certain classical disc functionals,  obtained 
by various authors in the manifold case (i.e., 
when the underlying space $X$ is nonsingular),
to complex spaces with singularities. 

One of the most important disc functionals
is the {\em Poisson functional}  which associates to an  
upper semicontinuous function $u$ on $X$ and 
an analytic disc $f\in\cA_X$ the average 
of the function $u\circ f$ over the circle $\bT=b\D$. A fundamental
result of Poletsky is that the envelope of the Poisson functional
on domains in $\C^n$ is always plurisubharmonic.
In \S\ref{sec:Poisson} we give another proof of our 
result from \cite{DF2} on plurisubharmonicity of 
Poisson functionals on any locally irreducible complex space, 
reducing it to Rosay's theorem in the manifold case by using 
Hironaka desingularization. The same proof also applies to
Riesz and Lelong functionals (see \S \ref{sec:Poisson}
for the definitions); so we obtain the following result:

\begin{thm}
\label{thm_functionals}
Let $X$ be an irreducible and locally irreducible complex space, 
and let $H_X \colon \cO(\clD,X) \to\ol\R$ 
be one of the following disc functionals:
\begin{itemize}
\item[\rm (i)]   $P_u$, the Poisson functional corresponding 
to an upper semicontinuous function $u$ on $X$ (see (\ref{eq:Poisson})
in \S\ref{sec:Poisson});
\item[\rm (ii)]  $R_u$, the Riesz functional corresponding 
to a plurisubharmonic function $u$ on $X$ (see (\ref{eq:Riesz}) in \S\ref{sec:Poisson});
\item[\rm (iii)] $L_{\alpha}$, the Lelong functional,  or $\wt L_\alpha$, 
the reduced Lelong functional, associated to a nonnegative 
function $\alpha$ on $X$ (see (\ref{eq:Lelong}) in \S\ref{sec:Poisson}).
\end{itemize}
Then the envelope $EH_X$ (\ref{eq:envelope}) is a
plurisubharmonic function on $X$.
\end{thm}

The assumption of local irreducibility can not be omitted. 
In \S\ref{sec:Poisson} we give an example of an 
irreducible complex curve with a single double point
such that the envelopes of the above functionals,
corresponding to appropriately chosen functions, 
are not plurisubharmonic at that point.

On a complex manifold $X$, the envelopes of the disc functionals mentioned above 
are the following extremal plurisubharmonic functions: 
\begin{itemize}
\item
The envelope $EP_u$ of the Poisson functional 
is the largest plurisubharmonic minorant of the 
upper semicontinuous function $u$.
\item  
The envelope $ER_u$ of the Riesz functional is
\[
	ER_u = \sup\{ v\in\Psh(X) \colon v\le 0,\ \d\d^c v\ge \d\d^c u\},
\]
the largest nonpositive plurisubharmonic function on $X$ 
whose Levi form is bounded below by the 
Levi form of $u$ \cite{Larusson-Sigurdsson1998}.
\item 
The envelope $EL_\alpha$ of the Lelong functional 
is the largest nonpositive plurisubharmonic function whose Lelong
number at each point $x\in X$ is $\ge \alpha(x)$ 
(see \S\ref{sec:Lelong} below).
\end{itemize}

In \S\ref{sec:Lelong} we give a new treatment 
of the Lelong functional, simplifying  
the proof of plurisubharmonicity of its envelope 
that was given by L\'arusson and Sigurdsson in 
\cite{Larusson-Sigurdsson1998,Larusson-Sigurdsson2003}.
The key point is obtained by the method of gluing 
holomorphic sprays of discs, 
similarly to what was done in \cite{DF2} 
for the Poisson functional. Our proof also applies 
to locally irreducible complex spaces without having 
to use the desingularization theorem.

In \S\ref{sec:Siciak} we find a formula expressing 
the Siciak-Zaharyuta  maximal function $V_{\Omega,X}$ 
of a nonempty open set $\Omega$ in a locally irreducible 
affine algebraic variety $X\subset\C^n$ as the envelope of 
appropriate Poisson functionals, obtained from Green functions
on complex curves in $X$ with boundaries in $\Omega$. 
For open sets in $X=\C^n$ such formulas 
have been obtained by Lempert (in the case when $\Omega$ is convex)
and by L\'arusson and Sigurdsson.

%
\section{Plurisubharmonicity of envelopes of disc functionals}
\label{sec:Poisson}
Let $X$ be a complex space. Given an upper semicontinuous function 
$u\colon X\to \R\cup\{-\infty\}$, the associated 
{\em Poisson functional}  is defined by
\begin{equation}
\label{eq:Poisson}
	P_u(f)= \frac{1}{2\pi} \int^{2\pi}_0 u(f(\E^{\I t}))\, \d t,
	\qquad f\in \cO(\clD,X).
\end{equation}

Let $u\colon X\to \R\cup\{-\infty\}$ be a plurisubharmonic function. 
The associated {\em Riesz functional}  is given by
\begin{equation}
\label{eq:Riesz}
	R_u(f)= \frac{1}{2\pi}  \int_{\D} \log|\cdotp| \, \Delta (u\circ f),
	\qquad f\in\cO(\clD,X). 
\end{equation}
The Laplacian $\Delta g$ of the subharmonic function 
$g=u\circ f$ is a positive Borel measure on $\clD$.
There is a close connection with the Poisson functional
which derives from the following {\em Riesz representation formula} on the disc:
\[
	g(0) = \frac{1}{2\pi} \int_0^{2\pi} g(\E^{\I t})\, \d t 
			+ \frac{1}{2\pi} \int_{\D} \log|\cdotp|\, \Delta g.
\]
Applying this to the function $g=u\circ f$ on $\ol\D$, where $u$ 
and $f$ are as in (\ref{eq:Riesz}), we obtain
\[
	u(f(0)) = P_u(f) + R_u(f).
\]
Setting $x=f(0)\in X$, this can be rewritten as 
$R_u(f)=u(x) + P_{-u}(f)$. Taking the infimum over all
$f\in\cO(\clD,X,x)$ yields  the following
relation between the Riesz and the Poisson envelope:
\begin{equation}
	ER_u = u + EP_{-u}. 
\label{eq_Riesz_Poisson}
\end{equation}
Therefore, to prove that $ER_u$ is plurisubharmonic, 
we need to show that $EP_{-u}<\infty$ and that $EP_{-u}$ 
is plurisubharmonic on $X$.

If $v\le 0$ is a plurisubharmonic function on $X$ such that 
$\d\d^c v\ge \d\d^c u$, then for every disc $f\in\cO(\ol \D,X)$ we have 
$\triangle (v\circ f)  \ge \triangle (u\circ f)$,
and hence $R_v(f)\le R_u (f)$.
Since $P_v(f)\le 0$, the Riesz formula gives 
\[
	v(f(0)) = P_v(f) + R_v(f) \le R_v(f) \le R_u(f).	 
\]
By taking the infimum over all discs with a given center we get
$v\le ER_u$. Once we know that $ER_u$ is plurisubharmonic,
it follows that it is the biggest plurisubharmonic 
function $v\le 0$ satisfying $\d\d^c v \ge \d\d^c u$.

The {\em Lelong functional} associated to a 
nonnegative real function $\alpha$ on $X$ is defined by
\begin{equation}
\label{eq:Lelong}
	L_\alpha(f)= \sum_{z\in \D} \alpha(f(z))\, m_f(z) \log|z|,
	\qquad f\in\cO(\clD,X),
\end{equation}
where $m_f(z)$ denotes the multiplicity of $f$ at $z$. 
We get the {\em reduced Lelong functional}, $\wt L_\alpha$, 
by removing the multiplicities $m_f(z)$ from the above formula.

Plurisubharmonicity of envelopes of these functionals 
on domains in $\C^n$ was established by Poletsky \cite{Poletsky1991,Poletsky1993};
similar results for the Poisson functional were found by 
Bu and Schachermayer \cite{Bu-Schachermayer}.
Poletsky's theorem was extended to all complex manifolds for the Poisson functional
by Rosay \cite{Rosay1,Rosay2} (see also Edigarian \cite{Edigarian2003-2}), 
and then for the other functionals mentioned above by L\'arusson and Sigurdsson 
\cite{Larusson-Sigurdsson1998,Larusson-Sigurdsson2003} and 
Edigarian \cite{Edigarian2003}.
The envelope of the Lelong functional coincides 
with the envelope of the corresponding 
reduced Lelong functional \cite{Larusson-Sigurdsson2003};
see Theorem \ref{thm:Lelong} below. 

In \cite{DF2} we proved that the envelope of the Poisson functional
is plurisubharmonic if  $X$ is a locally irreducible complex space.

%
%
%
%

\begin{thm}
\label{Poletsky-Rosay}
{\rm \cite[Theorem 1.1]{DF2}}
Let $X$ be a locally irreducible complex space
and $u\colon X\to \R\cup\{-\infty\}$ an upper semicontinuous 
function. Then the envelope 
\begin{equation}
\label{eq:Poisson-funct}
   \wh u(x) = \inf \Big\{\int^{2\pi}_0 u(f(\E^{\I t}))\, 
   \frac{\d t}{2\pi} \colon \ f\in \cO(\clD,X,x) \Big\},
    \quad x\in X,
\end{equation}
of the Poisson functional $P_u$ is the largest plurisubharmonic minorant of $u$. 
\end{thm}


The proof of this result given in \cite{DF2} is no more difficult 
than Rosay's proofs in \cite{Rosay1,Rosay2} for the case when $X$ is a complex manifold;
it combines Poletsky's proof on $X=\C^n$ with the method 
of gluing holomorphic sprays of discs. (For an exposition of the latter method 
we refer to \cite[\S 5.8--\S 5.9]{Fbook}.)
We use this opportunity to give another proof of Theorem \ref{Poletsky-Rosay}, 
reducing it to the case when $X$ is a complex manifold
by applying Hironaka desingularization theorem \cite{Hironaka,AHV,BM}.
The latter states that for every paracompact reduced complex space, $X$,
there is a proper holomorphic surjection $\pi\colon M\to X$ satisfying the
following properties: 
\begin{itemize}
\item $M$ is a complex manifold, 
\item $\pi\colon M\setminus \pi^{-1}(X_{\rm sing})\to X\setminus X_{\rm sing}$ 
is a biholomorphism, and 
\item $\pi^{-1}(X_{\rm sing})$ is a complex hypersurface in $M$.
\end{itemize}


We now prove the following more general result on envelopes of disc
functionals, showing that the only problem with plurisubharmonicity 
is at points where the complex space is locally reducible.

\begin{thm}
\label{desig_thm}
Let $X$ be a complex space and
let $\pi\colon M\to X$ be a desingularization of $X$.
Given a disc functional $H_X\colon \cO(\clD,X) \to\ol\R$,
we define a disc functional $H_M=\pi^* H_X\colon \cO(\clD,M) \to\ol\R$ by
\begin{align}
	H_M(g)=H_X(\pi \circ g) \hbox{ for each }g\in \cO(\clD,M).
  \label{eq_assthm}
\end{align}
If the envelope $EH_M$ is plurisubharmonic on $M$, then the envelope 
$EH_X$ given by (\ref{eq:envelope}) is plurisubharmonic on the regular 
part $X_{\rm reg}$ of $X$, and we have that
\begin{align}
	EH_X(x)=\inf\{EH_M(p)\colon p\in\pi^{-1}(x)\} \hbox{ for each }x\in X_{\rm sing}.
  \label{eq_bssthm}
\end{align}
If $X$ is locally irreducible at a point $x\in X_{\sing}$, 
then $EH_X$ is plurisubharmonic in a neighborhood of $x$ in $X$.
\end{thm}


We shall need the following lemma
on lifting holomorphic discs to a desingularization.
For the sake of completeness we include the proof.

\begin{lem}
\label{lifting}
Let $X$ be a complex space and
let $\pi\colon M\to X$ be a desingularization of $X$.
Given a holomorphic disc $f\in \cO(\clD,X)$ such that $f(\D)\not\subset X_{\sing}$
there exists a unique holomorphic disc 
$g\in \cO(\clD,M)$ such that $\pi\circ g=f$.
\label{lift_disc}
\end{lem}

\begin{proof}
Fix $f\in \cO(\clD,X)$ such that $f(\D)\not\subset X_{\sing}$.
Let $U_f$ be an open connected neighborhood of $\clD$ on which  
$f$ is holomorphic. Then $S=f^{-1}(X_{\sing})$ is a discrete
subset of $U_f$. Since the map $\pi\colon M\to X$ is 
biholomorphic on $M\setminus\pi^{-1}(X_{\sing})$,
there is a unique holomorphic map $g = \pi^{-1}\circ f \colon U_f\setminus S\to M$ 
such that $\pi\circ g=f$ on $U_f\setminus S$.
We need to show that $g$ extends holomorphically across $S$.
Pick a point $s\in S$ and a connected open neighborhood 
$U \subset U_f$ of $s$ such that $U\cap S=\{s\}$.
By shrinking $U$ around $s$ we may assume that the image 
$C=f(U)\subset X$ is an irreducible closed complex curve in an open 
neighborhood $V\subset X$ of the image point $f(s) = x \in X_{\sing}$.   
Its preimage $\pi^{-1}(C)$ is then a closed complex subvariety 
of the open set $\pi^{-1}(V)\subset M$. 
Observe that $\pi^{-1}(C) \setminus \pi^{-1}(x)=g(U\setminus\{s\})$.
Since the closure of the difference of two subvarieties
is again a subvariety (see e.g.\ \cite[Corollary, p.\ 53]{Chi}),
we infer that $\Sigma := \overline{g(U\setminus\{s\})}$ is a pure 
one dimensional complex subvariety of $\pi^{-1}(V)$ projecting onto $C$. 
(This is the proper transform of $C$ in $M$.) 
Since $g(U\setminus\{s\})$ is connected,  
the set $\Sigma \cap\pi^{-1}(x)$ consists of precisely one point, 
say $p$, and setting $g(s)=p$ extends the map $g$ 
holomorphically to the point $s$.
\end{proof}

\begin{proof}[Proof of Theorem \ref{desig_thm}]
\label{proof-by-desing}
Assume that the envelope $EH_M$ is plurisubharmonic. 
Since the map $\pi$ is biholomorphic over $X_{\rm reg}$, 
the function $EH_M|\pi^{-1}(X_{\rm reg})$ 
passes down to a plurisubharmonic function on $X_{\rm reg}$: 
\begin{equation}
\label{eq_EH_Mw}
			EH_M = w\circ \pi 	
\end{equation} 
for some function $w\colon X_{\rm reg}\to \R\cup\{-\infty\}$.
To see that $w=EH_X$ on $X_{\rm reg}$, choose a point $x\in X_{\reg}$ 
and let $p\in M$ be the unique point with $\pi(p)=x$. 
Every analytic disc $f\in \cO(\clD,X)$ with $f(0)=x$ 
satisfies $f(\D)\not \subset X_{\sing}$, and by Lemma \ref{lift_disc} it 
lifts to a disc $g\in \cO(\clD,M)$
centered at $g(0)=p$ so that $\pi\circ g=f$.
In particular, every disc $g\in \cO(\clD,M)$ with $g(0)=p\in \pi^{-1}(X_\reg)$
is the unique lifting of its projection $f=\pi\circ g$.
By taking the infimum over all discs $f\in \cO(\clD,X,x)$, 
the above implies in view of (\ref{eq_assthm}) and (\ref{eq_EH_Mw})
that $w(x)=EH_M (p)= EH_X(x)$. This shows that $w=EH_X$ on $X_{\reg}$.

Consider now a point $x \in X_{\sing}$. 
Since any disc $f\in\cO(\clD,X,x)$ such that $f(\D)\not\subset X_{\sing}$
lifts to a disc $g\in \cO(\clD,M)$ centered at some point $p\in\pi^{-1}(x)$, 
we get using (\ref{eq_assthm}) and (\ref{eq_EH_Mw}) that
$EH_M(p)\le H_M(g)=H_X(f)$.
By taking the infimum over all $f\in \cO(\clD,X,x)$ with 
$f(\D)\not\subset X_{\sing}$ we infer that 
\[
	\alpha:= \inf\{EH_M(p)\colon p\in\pi^{-1}(x)\}\le EH_X(x).
\]
To get the converse inequality,
pick $\epsilon>0$ and choose a point $p\in\pi^{-1}(x) \in M$
such that $EH_M(p)<\alpha +\epsilon$.
There is disc $g\in \cO(\clD,M,p)$ satisfying
\[
	H_M(g) <  EH_M(p)+\epsilon<\alpha+2\epsilon.
\]
By moving $g$ slightly, keeping its center fixed, we may 
assume that $g(\D)$ is not contained in the hypersurface $\pi^{-1}(X_{\sing})$.
Then the image of the disc 
$f=\pi \circ g \in \cO(\clD,X,x)$ is not contained in $X_{\sing}$. 
The above inequality together with (\ref{eq_assthm}) and (\ref{eq_EH_Mw}) implies 
\[
   H_X(f)= H_M(g) < EH_M(p)+\epsilon<\alpha+2\epsilon,
\]
and therefore $EH_X(x) <\alpha+2\epsilon$. Since $\epsilon>0$
was arbitrary, we have $EH_X(x)\le \alpha=\inf\{EH_M(p)\colon p\in\pi^{-1}(x)\}$
which proves (\ref{eq_bssthm}).

Assume now that $X$ is locally irreducible at a point $x \in X_{\sing}$.
Pick a point $p\in M$ with $\pi(p)=x$. 
Since $EH_M$ is upper semicontinuous and $\pi^{-1}(X_{\sing})$ 
is a hypersurface in $M$, we have
\[
	EH_M(p)=\limsup_{q\in \pi^{-1}(X_{\reg}),\ q\to p} EH_M(q).
\]
Local irreducibility of $X$ at the point $x$ implies that the fiber 
$\pi^{-1}(x)$ is a connected compact analytic set,
and therefore any plurisubharmonic function on $M$ 
is constant on $\pi^{-1}(x)$. This implies in view 
of (\ref{eq_bssthm}) that 
\[
	EH_X(x)=\limsup_{x'\in X_{\reg},\ x'\to x} EH_X(x').
\]
A theorem of Demailly \cite[Th\'eor\`eme 1.7]{Dem} now shows that the 
function $EH_X$, being plurisubharmonic on $X_{\reg}$, is 
also plurisubharmonic in a neighborhood of $x$ in $X$.
\end{proof}

Theorem \ref{thm_functionals} now follows from 
Theorem \ref{desig_thm} and the following lemma.

\begin{lem}
\label{desig_lemma}
Let $M$ and $X$ be complex spaces and $\pi\colon M\to X$ a holomorphic map.
Let $H_X\colon\cO(\clD,X) \to\ol\R$ be one of the following disc functionals:
\begin{itemize}
\item[\rm (i)]   $P_u$, the Poisson functional corresponding to an 
upper semicontinuous function $u$ on $X$,
\item[\rm (ii)]  $P_v$, the Poisson functional corresponding to a plurisuperharmonic function $v$ on $X$,
\item[\rm (iii)] $\wt L_{\alpha}$, the reduced Lelong functional associated to 
a nonnegative function $\alpha$ on $X$.
\end{itemize}
The disc functional $H_M\colon\cO(\clD,M) \to\ol\R$ defined by
\begin{align}
	H_M(g)=H_X(\pi \circ g),\qquad g\in \cO(\clD,M),
\end{align}
is then of the same kind (i)-(iii), and, if $M$ is nonsingular, 
then $EH_M$ is plurisubharmonic on $M$.
\end{lem}

\begin{proof}
If $H_X=P_u$ is the Poisson functional corresponding to an upper semicontinuous function $u$ on $X$,
then $H_M$ is the Poisson functional corresponding to the upper semicontinuous function $u\circ\pi$ on $M$, since
\[
	P_{u\circ\pi}(g)=\frac{1}{2\pi} \int^{2\pi}_0 u(\pi(g(\E^{\I t})))\, \d t
	=P_u(\pi \circ g) 
\]
for each $g\in \cO(\clD,M)$. A similar argument applies in case (ii).
%
If $H_X=\wt L_{\alpha}$ is the reduced Lelong functional associated to 
a nonnegative function $\alpha$ on $X$, then $H_M$ is the 
reduced Lelong functional corresponding to the nonnegative 
function $\alpha\circ\pi$ on $M$:
\[
	\wt L_{\alpha\circ\pi}(g)=\sum_{\zeta\in \D} \alpha(\pi(g(\zeta)))\, \log|\zeta|
	= \wt L_\alpha(\pi\circ g),\qquad g\in \cO(\clD,M).
\]

By the results cited above, the envelopes of all these disc functionals 
are pluri\-sub\-harmonic on $M$ if $M$ is smooth.
\end{proof}

\begin{rem}
 In this paper we defined the envelope of a disc functional at a point $x\in X$ 
 as the infimum over all discs centered at $x$ and not entirely
 lying in the singular locus $X_{\sing}$, whereas in \cite{DF2} we 
 considered the infimum with respect to {\em all discs} centered at $x$.
 The two envelopes clearly coincide on $X_{\reg}$, and if they are 
 plurisubharmonic then they coincide on all of $X$.
\end{rem}

%
%
%
%
\begin{exa}
\label{contr_examples}
We show that the conclusion of Theorem \ref{thm_functionals} 
fails in general at a point where $X$ is not locally irreducible. 
Such examples can already be found in the simplest case when $X$
is a complex curve with a single double point. 
To be explicit, consider the map $f\colon \C\to \C^2$ given by $f(z)=(z^3,z^2+z)$.
It is easily seen that $f$ is a proper holomorphic immersion
whose only double point is 
$p=f(\omega_1)=f(\omega_2)$, where $\omega_1\ne \omega_2$ 
are the two nonreal solutions of the equation $z^3=1$.
The immersed complex curve $X=f(3\D) \subset\C^2$ 
has exactly one self-intersection (at the point $p$), 
and $f\colon 3\D \to X$ is a desingularization of $X$.



Choose a smooth convex function $v\le 0$ on the disc $3\ol\D$ such that 
$v=0$ on $b(3\D)$ and $v(\omega_1)\ne v(\omega_2)$, 
and a linear function $A\colon\C\to \R$ such that
$v+A \le 0$ on $3\D$ and $(v+A)(\omega_1)= (v+A)(\omega_2)$.
It is easily seen that $EP_{-(v+A)}=-A$ on $3\D$.
The subharmonic function $v+A$ passes down to a  
smooth subharmonic function $v_X\le 0$ on $X$, and we get by
(\ref{eq_bssthm}) that
\[
	EP_{-v_X}(p) = \min \{EP_{-(v+A)}(\omega_1), EP_{-(v+A)}(\omega_2)\}
  =\min \{-A(\omega_1), -A(\omega_2)\}.
\]
Since these two values are different, we see that the Poisson envelope 
$EP_{-v_X}$ is not upper semicontinuous at the point $p$.
As $ER_{v_X}=v_X+EP_{-v_X}$ and $v_X$ is continuous on $X$, 
we also see that the envelope of the Riesz functional,
$ER_{v_X}$, is not upper semicontinuous at $p$.

For the Lelong functional, choose a function 
$\alpha\ge 0$ on $3\D$ with $\alpha(\omega_1)=\alpha(\omega_2)=0$ 
such that $EL_\alpha(\omega_1)\ne EL_\alpha(\omega_2)$. 
Then $\alpha$ passes down to a function $\alpha_X\ge 0$ on $X$ and 
$EL_{\alpha_X}(p)=\min\{EL_{\alpha}(\omega_1), EL_{\alpha}(\omega_2)\}$, 
so $EL_{\alpha_X}$ is not upper semicontinuous at $p$.
\end{exa}


\section{Envelopes of Lelong functionals}
\label{sec:Lelong}
In this section we give a new and simpler treatment of Lelong functionals
$(\ref{eq:Lelong})$ that have been considered 
earlier by Poletsky \cite{Poletsky1993,Poletsky2002}
any by L\'arusson and Sigurdsson \cite{Larusson-Sigurdsson1998,Larusson-Sigurdsson2003}. 
For simplicity we consider the case when  $X$ is a complex manifold, 
although the methods and results also apply on the regular locus 
of any complex space. (In the latter case, with $X$ locally irreducible
at each point, the proof of plurisubharmonicity of 
Lelong envelopes can be concluded as in Theorem \ref{desig_thm}.)

Given a plurisubharmonic function $u \in \Psh(X)$, 
we denote by $\nu_u(x) \in [0,+\infty]$ its Lelong number
at a point $x\in X$. Recall that in any local coordinate system 
$z$ on $X$, with $z(x)=0$, we have
\[
	\nu_u(x)= \lim_{r\to 0} \frac{\sup_{|z|\le r} u(z)}{\log r}.
\]
(We consider the function $u=-\infty$ as plurisubharmonic
and set $\nu_{-\infty}=+\infty$. Lelong numbers can also be defined 
for functions on complex spaces.)

Given a nonnegative function $\alpha\colon X\to \R_+$,  let
\[
		\cF_\alpha=\{u \in \Psh(X)\colon u\le 0,\ \nu_u \ge\alpha\}.
\]
The goal is to identify the corresponding maximal function 
\begin{equation}
\label{eq:maxFalpha}
	v_\alpha=\sup\{u\colon u\in \cF_\alpha\}
\end{equation}
as the envelope of certain disc functionals which arise
naturally from the following considerations. 
(Compare with \cite[\S 5]{Larusson-Sigurdsson1998}.)

If $u\in \cF_\alpha$ and $f\in \cA_X$, then clearly
$u\circ f\le 0$ is a subharmonic function on the disc $\D$ whose 
Lelong number at any point $z\in \D$ satisfies 
\[
	\nu_{u\circ f}(z) \ge \alpha(f(z))\, m_f(z),
\]
where $m_f(z)$ denotes the multiplicity of $f$ at the point $z$.
Hence $u\circ f$ is bounded above by the largest subharmonic function 
$v=v_{\alpha,f}\le 0$ on $\D$ satisfying $\nu_v\ge (\alpha\circ f)m_f$.
This maximal function $v$ is the weighted sum of 
Green functions with coefficients $(\alpha\circ f)m_f$: 
\[
	v(\zeta) = \sum_{z\in \D} \alpha(f(z)) \, m_f(z) 
		\log\left| \frac{\zeta-z}{1-\bar z\zeta}\right|,
		\quad \zeta\in\D.
\]
(If the sum is divergent then $v\equiv-\infty$.)
Indeed, the difference between $v$ and the right hand side above
is subharmonic on $\ol\D$, except perhaps at the points $z$ where 
$\alpha(f(z))>0$; near these points it is bounded above,
so it extends to a subharmonic function on $\D$.
Since it is clearly $\le 0$ on $b\D$, the maximum principle implies that
it is $\le 0$ on all of $\D$ which proves the claim.

Setting $\zeta=0$, we see that for every $u\in \cF_\alpha$ and 
$f\in \cA_X$ we have 
\[
	u(f(0)) \le \sum_{z\in \D} \alpha(f(z)) m_f(z) \log|z| 
	\le  \sum_{z\in \D} \alpha(f(z)) \log|z| 
	\le \inf_{z\in \D} \alpha(f(z)) \log|z|.
\]
(The second and the third inequality are trivial.)
These expressions determine the following disc functionals on $X$
with values in $[-\infty,0]$:
\begin{eqnarray*}
	L_\alpha(f) &=& \sum_{z\in \D} \alpha(f(z))\, m_f(z) \log|z|, \cr
	\wt L_\alpha(f) &=& \sum_{z\in \D} \alpha(f(z))\, \log|z|, 
	\label{eq:Lelongred} \cr
	K_\alpha(f) &=& \inf_{z\in \D} \alpha(f(z)) \log|z|.
	\label{eq:K} 
\end{eqnarray*}
The first two are the {\em Lelong functional}, $L_\alpha$,
and the {\em reduced Lelong functional}, $\wt L_\alpha$,
that have already been mentioned in \S\ref{sec:Poisson}.
By taking infima over all analytic discs $f$ in $X$ with a fixed
center $f(0)=x\in X$ we obtain the corresponding inequalities
for their envelopes:
\begin{equation}
\label{eq:ineq-envelopes}
	u\le EL_\alpha \le E\wt L_\alpha \le EK_\alpha =: k_\alpha,
	\quad \forall u\in \cF_\alpha. 
\end{equation}
The function $k_\alpha=EK_\alpha \colon X\to [-\infty,0]$,
which is denoted $k^\alpha_X$ in \cite[p.\ 21]{Larusson-Sigurdsson1998},
is related to a certain function studied by Edigarian \cite{Edigarian1997}.
It is easily seen that $k_\alpha$ is upper semicontinuous 
(it suffices to move the center $f(0)$ of the test disc, 
while at the same time fixing the value $f(z)\in X$ at a point $z\in \D$ 
where  $\alpha(f(z)) \log |z|$ is close to optimal), and that 
its Lelong numbers are bounded below by $\alpha$ 
\cite[Proposition 5.2]{Larusson-Sigurdsson1998}.
Hence the Poisson envelope $EP_{k_\alpha}$
(which is the largest plurisubharmonic minorant of $k_\alpha$
according to Theorem \ref{Poletsky-Rosay})
also has Lelong numbers bounded below by $\alpha$, and so it
belongs to the class $\cF_\alpha$. Since $u\le k_\alpha$ 
for every $u\in\cF_\alpha$ by (\ref{eq:ineq-envelopes}), 
it follows that $EP_{k_\alpha}=v_\alpha$ is the maximal
function (\ref{eq:maxFalpha}). Furthermore, and this is the only
nontrivial thing that remains to be seen, the envelopes 
$EL_\alpha$ and $E\wt L_\alpha$ also equal the maximal 
function $v_\alpha$. 

Summarizing the above discussion, we have the following result.

\begin{thm} 
\label{thm:Lelong}
For every function $\alpha\ge 0$ on a complex manifold $X$
the maximal function $v_\alpha$ (\ref{eq:maxFalpha}) is 
plurisubharmonic and equals the envelope of both the
Lelong and the generalized Lelong functionals:
\begin{equation}
\label{eq:equalities}
		v_\alpha= EL_\alpha = E\wt L_\alpha = EP_{k_\alpha}.  
\end{equation}
If $X$ is a complex space then these equalities hold
on the regular locus $X_{\reg}$; if $X$ is locally irreducible 
at every point, then they hold on all of $X$.
\end{thm}

On manifolds this was proved by L\'arusson and Sigurdsson, 
first for domains in Stein manifolds \cite{Larusson-Sigurdsson1998},
and then, following the work of Rosay \cite{Rosay1,Rosay2},
on all complex manifolds \cite{Larusson-Sigurdsson2003}.
We find their proof rather difficult even for domains in Stein manifolds.
Here we give a direct proof of the equalities (\ref{eq:equalities}) 
on the regular locus of any complex space.  
On locally irreducible complex spaces the result then follows 
by the arguments in \S\ref{sec:Poisson}.

We shall need the following elementary lemma.

\begin{lem}
\label{lemma1}
Let $J$ be a union of finitely many closed arcs in the circle $\bT=b\D$,
let $U\subset \C$ be an open set containing $J$, and let
$\zeta\colon U\cap \clD\to \C$ be a continuous function
satisfying $0<|\zeta(z)|<1$ for $z\in U\cap \clD$.
Given $\epsilon>0$, the following inequality 
holds for all sufficiently large integers $k\in \N$:
\begin{equation}
\label{eq:estimate1}
	\sum_{z\in U\cap \D,\ z^k=\zeta(z)} \log|z| < 
	\int_{\E^{\I t}\in J} \log|\zeta(\E^{\I t})| \, \frac{\d t}{2\pi} + \epsilon.
\end{equation}
\end{lem}

\begin{proof}
This is obvious when $\zeta(z)=a$ is a constant function.
In that case the equation  $z^k=a$ has $k$ solutions 
\[
	z_j(a) = \sqrt[k]{|a|}\,\, \E^{\I(\theta_0+j/2k\pi)}, \qquad j=1,\ldots,k,
\]
where $a=|a|\,\E^{\I k\theta_0}$. 
Since the points $z_1(a),\ldots,z_k(a)$ are equidistributed along the circle
$|z|=\sqrt[k]{|a|}$, at least $k|J|$ of them belong
to $U\cap \D$ if $k$ is big enough. 
(Here $|J|$ denotes the normalized arc length of $J$.)
This gives
\[	 
		\sum_{z_j(a)\in U\cap \D} \log|z_j(a)| \le 
		k |J| \log \sqrt[k]{|a|} = \log |a|\cdotp |J|
		= \int_{\E^{\I t}\in J} \log|a| \, \frac{\d t}{2\pi}.
\]

In the general case we break $J$ into
pairwise disjoint closed segments $J_1,\ldots,J_m$
(separated by short gaps) such that, for some choice 
of points $\E^{\I t_j} \in J_j$ and open pairwise 
disjoint sets $U_j$ with $J_j \subset U_j\Subset U$,
we have 
\begin{equation}
\label{eq:estimate2}
		\sum_{j=1}^m \log|\zeta(\E^{\I t_j})| \cdotp |J_j| 
		< \int_{\E^{\I t}\in J} \log|\zeta(\E^{\I t})| \, \frac{\d t}{2\pi} 
		+ \frac{\epsilon}{2}
\end{equation}
and 
\begin{equation}
\label{eq:estimate3}
		\left| \zeta(z) - \zeta(\E^{\I t_j}) \right| <
		\frac{\epsilon}{2}\, |\zeta(\E^{\I t_j})|,
		\quad z\in U_j\cap \clD.
\end{equation}
It suffices to show that for every $j=1,\ldots, m$ we have that
\begin{equation}
\label{eq:estimate4}
  	\sum_{z\in U_j,\ z^k=\zeta(z)} \log|z| < 
  	\log |\zeta(\E^{\I t_j})| \cdotp |J_j| + \frac{\epsilon}{2}\, |J_j|.
\end{equation}
Indeed, by summing the inequalities (\ref{eq:estimate4})
over all $j=1,\ldots,m$ and using 
(\ref{eq:estimate2}) we obtain the estimate (\ref{eq:estimate1}).

We now prove (\ref{eq:estimate4}). Fix $j\in\{1,\ldots,m\}$ and write 
$a_j= \zeta(\E^{\I t_j})$, so $0<|a_j|<1$. 
Let $\Delta_j\subset \C$ be the open disc of radius 
$r_j=\frac{\epsilon}{2}|a_j|$ centered at $a_j$.
By choosing $\epsilon>0$ small enough we may assume that 
$\ol\Delta_j \subset \D^*=\D\setminus \{0\}$ for each $j$. 
Since the map $\D^*\to \D^*$, $z\mapsto z^k$, is a $k$-fold 
covering, the preimage of $\ol\Delta_j$
is a disjoint union of $k$ simply connected closed domains (discs)
$\ol\Delta_{j,l} \subset \D^*$, $l=1,\ldots,k$. 
As $k\to +\infty$, the discs $\Delta_{j,l}$ converge to 
the circle $\bT$ and are equidistributed around $\bT$.
For $k$ big enough at least the proportional number
$k|J_j|$ of the discs $\ol \Delta_{j,l}$ are contained in $U_j\cap \D$.
Let $\Delta_{j,l}$ be such a disc. 
As the point $z$ traces the boundary $b\Delta_{j,l}$ 
in the positive direction, the image point $z^k$ traces 
$b\Delta_j$ once in the positive direction. 
From the estimate (\ref{eq:estimate3})
we infer that the function $z\mapsto z^k-\zeta(z)$ has winding number one
around $b\Delta_{j,l}$, and hence the equation $z^k=\zeta(z)$
has a solution $z=z_{j,l}$ in $\Delta_{j,l}$.
(If the function $z\mapsto \zeta(z)$ is holomorphic, as will
be the case in our application, then there is precisely one
solution in  $\Delta_{j,l}$ by Rouch\'e's theorem.) 
Clearly this solution satisfies 
\[
	\log|z_{j,l}|= \frac{1}{k}\log|\zeta(z_{j,l})| < 
	\frac{1}{k} \log \left( |a_j|(1+\frac{\epsilon}{2}) \right)
	< \frac{1}{k} \left( \log|a_j| + \frac{\epsilon}{2} \right).
\]
Since there are at least $k|J_j|$ solutions $z_{j,l}$,
the sum of their logarithms is bounded above 
by $|J_j|\left( \log|a_j| + \frac{\epsilon}{2} \right)$ which 
gives (\ref{eq:estimate4}). 
\end{proof}

\noindent {\em Proof of Theorem \ref{thm:Lelong}.}
In view of (\ref{eq:ineq-envelopes}) and the equality
$\sup_{u\in \cF_\alpha} u = EP_{k_\alpha}$ 
(see the paragraph preceding Theorem \ref{thm:Lelong}) we have that
\[
	EP_{k_\alpha} \le EL_\alpha \le E\wt L_\alpha \le k_\alpha.
\]
To establish (\ref{eq:equalities})  it remains 
to prove that $E\wt L_\alpha \le EP_{k_\alpha}$.
Equivalently, we need to show that for every continuous function 
$\phi \colon X\to \R$ with $\phi\ge k_\alpha$, analytic disc $h\in \cA_X$,
and number $\epsilon>0$ there exists  a disc 
$f\in \cA_X$ such that $f(0)=h(0)$ and
\begin{equation}
\label{eq:mainestimate}
   \wt L_\alpha(f) = \sum_{z\in \D} \alpha(f(z)) \, \log|z| 
   < \frac{1}{2\pi} \int_0^{2\pi} (\phi\circ h)(\E^{\I t}) \, \d t  +\epsilon. 
\end{equation}

The definition of the Poisson envelope $EP_{k_\alpha}$ of the function 
$k_\alpha$ shows that for every fixed $t_0\in \R$ there exist an analytic disc
$g_0\in \cA_X$ and a point $b_0\in \D$ such that 
$g_0(0)=h(\E^{\I t_0})=:x_0$ and 
\[
	\alpha(g_0(b_0)) \log|b_0| < \phi(x_0) +\frac{\epsilon}{2}.
\]
We embed $g_0$ into a family (spray) of analytic discs $g_x=g(x,\cdotp)\in \cA_X$,
depending holomorphically on the point $x$ in an open neighborhood
$V\subset X$ of $x_0$, such that $g_0=g(x_0,\cdotp)$ is the initial disc,
and for all $x\in V$ we have $g(x,0)=x$ and $g(x,b_0)=g(x_0,b_0)=:y_0$. 
By continuity there is a nontrivial closed arc $J\subset \bT$
around the point $\E^{\I t_0}$ such that $h(J)\subset V$ and 
\[
	\alpha(y_0)\, \log|b_0| \cdotp |J| < 
	\int_J (\phi\circ h)(\E^{\I t}) \frac{\d t}{2\pi} + \frac{\epsilon}{2} |J|.
\]
Repeating this argument at other points of the circle $\bT$
we find 
\begin{itemize}
\item pairwise disjoint closed arcs $J_1,\ldots, J_m\subset \bT$
with arbitrary short gaps between them,
\item points $\E^{\I t_j}\in J_j$, 
\item open sets $V_1,\ldots, V_m\subset X$ with $h(J_j)\subset V_j$ for $j=1,\ldots,m$, 
\item holomorphic sprays of discs $g_j\colon V_j\times\clD \to X$, and
\item points $b_1,\ldots,b_m\in \D$, 
\end{itemize}
such that the following properties hold:
\begin{itemize}
\item[\rm (a)] $g_j(x,0)=x$ for all $x\in V_j$ and $j=1,\ldots,m$,
\item[\rm (b)] the point $y_j=g_j(x,b_j)\in X$ is independent of $x\in V_j$, and
\end{itemize}
\begin{equation}
\label{eq:estimate-c} 
		\sum_{j=1}^m \alpha(y_j) \log |b_j|\cdotp |J_j| 
		< \int_0^{2\pi} (\phi\circ h)(\E^{\I t}) \, \frac{\d t}{2\pi} 
		+ \frac{\epsilon}{2}. 
\end{equation}
The integral of $\phi\circ h$ over $\bT\setminus \cup_{j=1}^m J_j$
is made small by choosing the arcs $J_j$ such that the measure 
of the complement is sufficiently small.
We may assume that $\alpha(y_j)>0$ for each $j$, for otherwise 
the corresponding term and the arc $J_j$ may be deleted.

Choose smoothly bounded simply connected sets (discs)
$\Delta_1,\ldots,\Delta_m$ in $\D$ with pairwise disjoint closures
such that $b\Delta_j\cap \bT$ contains a relative neighborhood of the arc
$J_j$ in the circle $\bT=b\D$. By choosing the discs $\Delta_j$
small enough we can also insure that $h(\ol\Delta_j)\subset V_j$.
Pick a larger arc $J'_j\subset \bT$ such that 
$J_j\Subset J'_j \Subset b\Delta_j\cap \bT$.
Set $D_1=\bigcup_{j=1}^m \Delta_j$.
Let $D_0\subset \D$ be a domain obtained by denting the circle $\bT$ 
slightly inward along each of the arcs $J'_j$ so as to insure
that $\ol D_0 \cap J'_j=\emptyset$ for all $j=1,\ldots,m$,
$D_0\cup D_1 =\D$, and 
$\ol{D_0\setminus D_1} \cap \ol{D_1\setminus D_0} =\emptyset$.
Thus $(D_0,D_1)$ is a Cartan pair in the sense of \cite[\S 5.7]{Fbook}.
The configuration around the disc $\Delta_j$ 
is shown on Fig.\ \ref{Fig1}.

%
%
%
%
\begin{figure}[ht]
\psset{unit=0.6cm, linewidth=0.6pt}  
\begin{pspicture}(-8,0)(8,7)

\psarc[linewidth=0.4pt](0,-4){10}{35}{145}
\psarc[linewidth=1.2pt,arrows=*-*](0,-4){10}{63}{117}
\psecurve(-5,4.7)(-6,4)(-6.5,3)(-4,2.5)(0,4) (4,2.5)(6.5,3)(6,4)(5,4.7)
\psarc[linewidth=1pt,linestyle=dotted](0,-10){15.3}{68}{112}

\rput(0,6.6){$J'_j$}
\rput(4.5,3.3){$\Delta_j$}
\rput(0,2){$\D$}
\rput(0,4.8){$bD_0$}
\psline[linewidth=0.2pt]{->}(0.6,4.8)(1.7,5.1)
\psline[linewidth=0.2pt]{->}(-0.6,4.8)(-1.7,5.1)

\end{pspicture}
\caption{The disc $\Delta_j$ and the arc $J'_j$}
\label{Fig1}
\end{figure}
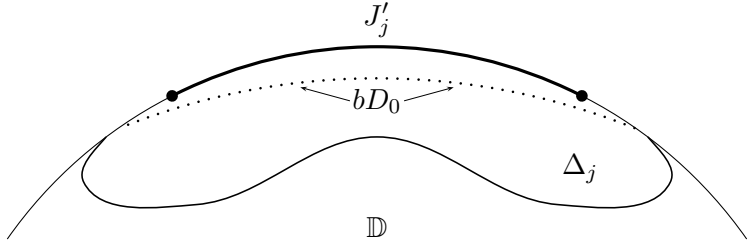

Pick a large constant $M>0$ whose precise value will be fixed later.
For every $j=1,\ldots,m$ let $u_j\le 0$ be a smooth real function on  
$\ol\Delta_j$, harmonic in $\Delta_j$, such that $u_j=0$
on $J_j$, and $u_j=-M$ on $b\Delta_j\setminus J'_j$.
Let $v_j$ be a harmonic conjugate of $u_j$, and define the function
$a_j\colon \ol\Delta_j\to \C$ by
\[
	a_j(z)= z^k \E^{u_j(z)+\I v_j(z)}, \quad z\in \ol\Delta_j.
\]
The value of the integer $k\in \N$ will be fixed later.
Note that $|a_j(z)|\le \E^{u_j(z)}$ for every $k$, and 
this is uniformly as close to zero as desired outside of any 
neighborhood of $J'_j$ if the constant $M>0$ is big enough.

Let $z_{j,1},\ldots,z_{j,l_j}\in \Delta_j$ be all the solutions
of the equation $a_j(z)=b_j$ in the disc $\Delta_j$. 
(Recall that the number $b_j\in \D^*$ satisfies property (b) above.)
This equation can be rewritten as 
\[
	z^k=b_j \E^{-u_j(z)-\I v_j(z)}=: \zeta_j(z).
\]
Note that $|\zeta_j(\E^{\I t})|=|b_j|$ for $\E^{\I t} \in J_j$.
By Lemma \ref{lemma1} we have for all large enough $k\in \N$ that  
\begin{equation}
\label{estimate6}
		\sum_{l=1}^{l_j} \log|z_{j,l}| < 
		\int_{J_j} \log|\zeta_j(\E^{\I t})| \, \frac{\d t}{2\pi} 
		+ \frac{\epsilon |J_j|}{2\alpha(y_j)} 
		= \log|b_j| \cdotp |J_j| + \frac{\epsilon |J_j|}{2 \alpha(y_j)}. 
\end{equation}

To complete the proof we now construct 
a holomorphic disc $f\in \cA_X$ with $f(0)=h(0)$ such that
\begin{equation}
\label{eq:f}
	f(z_{j,l}) = y_j,\qquad l=1,\ldots,l_j,\ j=1,\ldots,m.
\end{equation} 
For such $f$ we obtain by combining the inequalities 
(\ref{eq:estimate-c}) and (\ref{estimate6}) that
\begin{eqnarray*}
	\wt L_\alpha(f) &\le & 
	\sum_{j=1}^m \alpha(y_j) \sum_{l=1}^{l_j} \log|z_{j,l}| \cr
	&\le & \sum_{j=1}^m \alpha(y_j)\log|b_j| \cdotp |J_j| + \frac{\epsilon}{2} 
	 <  \int_0^{2\pi} (\phi\circ h)(\E^{\I t}) \, \frac{\d t}{2\pi}	+ \epsilon, 
\end{eqnarray*}
so the estimate (\ref{eq:mainestimate}) holds and the proof is complete.

To find such a disc $f$ we proceed as follows. 
We embed the disc $h$ into a dominating spray of discs
$h_w=h(w,\cdotp)\in \cA_X$, depending holomorphically on the 
point $w$ in a ball $\B$ in a Euclidean space $\C^N$, so that $h_0=h$.
By shrinking $\B$ around the origin we may assume that
$h(w,z)\in V_j$ holds for every $w\in \B$ and $z\in \ol\Delta_j$.
Over each of the discs $\ol\Delta_j$ we define a holomorphic spray 
of discs, with the parameter $w\in \B$, by setting
\[
	\wt g(w,z)= g_j\bigl( h(w,z),a_j(z)\bigr),
	\qquad z\in \ol\Delta_j.
\]
Since $a_j(z_{j,l})=b_j$ and $g_j(x,b_j)=y_j$ for all $x\in V_j$, we have  
$ 
	\wt g(w,z_{j,l}) = y_j
$ 
for all $w\in \B$, $j=1,\ldots,m$ and $l=1,\ldots,j_l$.
Further, for points $z\in \ol\Delta_j\cap \ol D_0$ 
the function $|a_j(z)|$ can be made arbitrary small 
by choosing the constant $M>0$ in the above construction big enough 
(and this estimate is independent of the choice of the integer $k\in \N$). 
For such $z$ we have 
\[
	\wt g(w,z) \approx g(h(w,z),0) = h(w,z).
\]
In particular, for $M>0$ big enough (and for any $k\in \N$) 
the sprays $h(w,z)$ (over $z\in  \ol D_0$) and $\wt g(w,z)$ 
(over $z\in \ol D_1$) can be glued into a single spray 
over $\ol D_0\cup \ol D_1=\ol \D$ by using the method from
\cite[\S 5.8 -- \S 5.9]{Fbook}.
(A brief exposition can also be found in \cite{DF2}.) 
Explicitly, we find maps $\beta_0\in \cA(D_0,\C^N)$
and $\beta_1\in \cA(D_1,\C^N)$, with values in $\B$, that 
are uniformly small (depending only on the uniform distance between 
the two sprays over $\ol D_0\cap \ol D_1$, and hence only on 
the constant $M>0$), such that $\beta_0(0)=0$ and
\[	
	\wt g(\beta_1(z),z) = h(\beta_0(z),z),\quad z\in \ol D_0\cap \ol D_1.
\] 
The two sides define an analytic disc $f\in \cA_X$ with center $f(0)=h(0)$ and
satisfying (\ref{eq:f}). This completes the proof of Theorem \ref{thm:Lelong}.
\qed


\section{Siciak-Zaharyuta extremal functions on affine varieties}
\label{sec:Siciak}
In this section we obtain explicit expressions for the Siciak-Zaharyuta 
extremal function $V_{\Omega,X}$ of an open set $\Omega$ in a
locally irreducible affine algebraic variety $X\subset \C^n$ 
in terms of Green functions on complex curves $C$ in the projective 
closure $\ol X\subset \P^n$ of $X$, with smooth boundaries 
$bC$ contained in $\Omega$. Theorem \ref{Maximalfunction-variety} 
below generalizes some of the results of Lempert and of
L\'arusson and Sigurdsson (for the case $X=\C^n$)  mentioned below.  

We begin by recalling some standard notions of pluripotential theory,
referring to Klimek \cite{Klimek} for further information. 

The {\em Lelong class} $\cL=\cL_{\C^n}$ on $\C^n$ 
is the set of all plurisubharmonic functions $v\colon \C^n\to\R\cup\{-\infty\}$
for which there exist constants $r>0$ and $C\in \R$ 
(depending on $v$) such that
\[
	v(z) \le \log |z| + C, \qquad z\in \C^n, \ |z|>r.
\]
Such function is said to have at most logarithmic growth at infinity.

Given a nonempty open subset $\Omega\subset \C^n$, the
{\em Siciak-Zaharyuta maximal function} $V_\Omega\colon\C^n\to \R$
is defined by 
\[
	V_\Omega(z) = \sup\{v(z)\colon v\in\cL,\ v|_\Omega\le 0\}.
\] 
The function $V_\Omega$ is the largest  function in the Lelong class 
$\cL$ which is $\le 0$ on $\Omega$.
By replacing any function $v\in \cL$ in the definition of $V_\Omega$ by $\max\{v,0\}$  
we see that $V_\Omega\ge 0$ on $\C^n$ and $V_\Omega=0$ on $\Omega$.
If $\Omega$ has a sufficiently nice boundary then 
$V_\Omega$ is continuous, and hence it vanishes on $\ol\Omega$.
(The extremal function $V_E$ can be defined for an arbitrary subset $E\subset \C^n$,
but in general one must take its upper regularization $V^*_E$ in order to get 
a plurisubharmonic function. We have $V^*_E\equiv +\infty$ if and only if 
the set $E$ is pluripolar. Here we restrict our attention to maximal functions
of open sets.) 

In the same way one defines the Lelong class $\cL_X$, and the maximal 
function $V_{\Omega,X}$, when $\Omega$ is an open subset in a closed
affine algebraic subvariety $X$ of a complex affine space $\C^n$.

We now recall the Lempert formula for the extremal function 
$V_\Omega$ of an open convex set $\Omega\subset \C^n$ 
(see the appendix in \cite{Momm} and the discussion in 
\cite{Larusson-Sigurdsson2005}). Consider $\C^n$ as a subset of the projective
space $\P^n$, and let $H=\P^n\setminus \C^n \cong\P^{n-1}$ denote 
the hyperplane at infinity. For every analytic disc $f\in \cA_{\P^n}$
with $f(\bT)\subset \C^n$ set
\begin{equation}
\label{eq:Jf}
		J(f)= -\sum_{f(\zeta)\in H} \log|\zeta| \ge 0,
\end{equation}
where the sum is over the finitely many points $\zeta\in \D$ which
are mapped by $f$ to $H$, counted with intersection multiplicities.
(If $f(0)\in H$, we set $J(f)=+\infty$, and if $f(\clD)\subset \C^n$
then $J(f)=0$.)  	We may think of $J$ as a disc functional on the set of discs
in $\P^n$ with boundary values in $\C^n$. Lempert proved that,
if $\Omega$ is open and convex, then for every point $z\in \C^n$ we have
\begin{equation}
\label{eq:Lempert-formula}
	V_\Omega(z) = \inf \left\{J(f) \colon 
	f\in \cA_{\P^n},\ f(0)=z,\ f(\bT)\subset \Omega\right\};
\end{equation}
furthermore, one gets the same infimum over the smaller set of discs 
with a single point at infinity of multiplicity one.

Lempert's formula (\ref{eq:Lempert-formula}) was extended  
by L\'arusson and Sigurdsson to the case when $\Omega$ is an arbitrary 
{\em connected} open subset of $\C^n$ \cite[Theorem 3]{Larusson-Sigurdsson2005};
however, one must in general use discs with several poles.

If $\Omega$ is disconnected, then the infimum on the right hand side
of (\ref{eq:Lempert-formula}) is in general larger 
than $V_\Omega(z)$.
However, it is still possible to obtain $V_\Omega$ as follows.
Let $\cB$ be a family of analytic discs in $\P^n$ with boundary 
values in $\Omega$. Following \cite{Larusson-Sigurdsson2005}
we introduce the following notion. 

%
%
%
%
\begin{defin}
\label{def:good}
A family of discs $\cB \subset \cA(\D,\P^n)$ with boundaries
in $\Omega\subset \C^n$ is a {\em good family} 
(with respect to $\Omega$) if it satisfies the following properties:
\begin{itemize}
\item[\rm (i)]   for every point $z\in \P^n$ there is a disc $f\in \cB$ with $f(0)=z$,
\item[\rm (ii)]   for every point $z\in \Omega$ the constant disc $\clD\to z$ belongs to $\cB$, and
\item[\rm (iii)] for every point $p\in \P^n$ and every disc $f\in \cB$ with $f(0)=p$
there exist a neighborhood $U\subset \P^n$ of $p$ and a continuous family 
of discs $\{f_z\in \cB\colon z\in U\}$ such that $f_{p}=f$ and $f_z(0)=z$ for all $z\in U$.
\end{itemize}
\end{defin}

Define a function $E_\cB J \colon \C^n\to \R_+$ by setting
\begin{equation}
\label{eq:EBJ}
	E_{\cB} J (z) = \inf_{f\in \cB,\ f(0)=z} J(f),\quad z\in\C^n.
\end{equation}
We think of $E_\cB J$ as the envelope of the disc functional 
$f\mapsto J(f)$ with respect to the family $\cB$. 
It is easily seen that the envelope $E_\cB J$ with respect to 
a good family of discs is an upper semicontinuous
function on $\C^n$ which vanishes on $\Omega$ and has at most logarithmic
growth at infinity (see \cite{Larusson-Sigurdsson2005}).

%
%
%
%
\begin{thm}
\label{LSThm2}
{\rm \cite[Theorem 2]{Larusson-Sigurdsson2005}}
If $\Omega$ is a nonempty open set in $\C^n$ and $\cB$ is a good family
of analytic discs with respect to $\Omega$, then the 
Siciak-Zaharyuta  function $V_\Omega$ is the Poisson envelope
of the function $E_\cB J$ (\ref{eq:EBJ}):
\begin{equation}
\label{eq:Lempert-env}
	V_\Omega(z) = \inf \Big\{\int^{2\pi}_0 (E_\cB J)(f(\E^{\I t}))\, 
   \frac{\d t}{2\pi} \colon \ f\in \cA(\D,\C^n,z) \Big\},
    \quad z\in \C^n. 
\end{equation}
\end{thm}

\begin{rem}
\label{on-Lemperts-formula}
If the set $\Omega$ is connected, then the Lempert formula (\ref{eq:Lempert-formula})
follows by applying Theorem \ref{LSThm2}, with $\cB$ a suitable family of discs 
in projective lines (and with boundaries in $\Omega$), and then 
solving a Riemann-Hilbert boundary value problem.
(See the last section in \cite{Larusson-Sigurdsson2005}.) 
The advantage of the formula (\ref{eq:Lempert-formula}) over 
(\ref{eq:Lempert-env})  is that the first one expresses the 
extremal function $V_\Omega$ by only one application of infimum 
over a suitably large family of discs, while on the latter the infimum
is applied twice.
\qed\end{rem}

L\'arusson and Sigurdsson proved Theorem \ref{LSThm2} by  
lifting the problem with respect to the projection
$\pi\colon \C^{n+1}\setminus\{0\} \to \P^n$ and
considering the plurisubharmonic function 
$V_\Omega\circ\pi + \log |z_0|$ on $\C^{n+1}\setminus\{0\}$, 
where the coordinates $(z_0,\ldots,z_n)$ on $\C^{n+1}$ 
are chosen such that $H=\{z_0=0\}$.
An application of the Riesz formula on the disc leads to an 
auxiliary disc formula \cite[Theorem 1]{Larusson-Sigurdsson2005}
from which the result is obtained by some additional arguments.

We now give a very simple proof of Theorem \ref{LSThm2}
which generalizes immediately to open sets in affine varieties.
The main point is to observe that the restriction of $V_\Omega$
to any complex curve $C\subset \P^n$ with smooth boundary contained
in $\Omega\subset \C^n$ is bounded above by the Green function on that curve
with poles at the points in $C\cap H$. 
(In the case at hand we can use discs, and this
brings the functional $J(f)$ into the picture.)
The infimum of such Green functions over sufficiently many curves 
yields an upper semicontinuous function $\Psi\colon \C^n\to\R_+$ 
that vanishes on $\Omega$, has at most logarithmic growth at infinity, 
and satisfies $V_\Omega\le \Psi$. 
It then follows from maximality that $V_\Omega$ is the Poisson 
envelope $EP_\Psi$ of $\Psi$.

\smallskip
\noindent{\it Proof of Theorem \ref{LSThm2}.}
Let $\Omega\subset \C^n\subset \P^n$ be as in the theorem.
Assume that $\ol \Sigma$ is a compact finite bordered 
Riemann surface all of whose boundary components are Jordan curves.
Let $f\colon \ol\Sigma \to \P^n$ be a continuous map that is 
holomorphic in the interior, $\Sigma$, of $\ol \Sigma$,
and with boundary $f(b\Sigma)\subset \Omega$ contained in $\Omega$.
Recall that $H=\P^n\setminus \C^n$.
Write $f^{*}(H)=\sum_{j=1}^k m_j p_j$ as a divisior,
where $p_1,\ldots,p_k$ are the  points in 
$\Sigma$ that are mapped by $f$ to the hyperplane $H$, 
and $m_j\in\N$ is the intersection multiplicity of 
the map $f$ with $H$ at the point $p_j$. The composition 
\[
	V_\Omega\circ f\colon \ol\Sigma\to\R_+ \cup\{+\infty\}
\]
is then a subharmonic function on $\ol\Sigma\setminus f^{-1}(H)$ 
that vanishes on $b\Sigma$ and has logarithmic poles 
at the points $p_j\in f^{-1}(H)$. More precisely, 
choosing a local holomorphic coordinate $\zeta$ on $\Sigma$
with $\zeta(p_j)=0$, there is a constant $C\in\R$ such that 
$V_\Omega\circ f(\zeta) \le - m_j \log|\zeta| + C$ as $\zeta\to 0$.
It follows that 
\[
	V_\Omega\circ f\le - G_{\Sigma,f^*H}= - \sum_{j=1}^k m_j G_{\Sigma,p_j},
\] 
where the right-hand side is the Green function on $\Sigma$ with poles 
determined by the divisor $f^* H$.  (Precisely, $G_{\Sigma,f^*H}\le 0$	 
is the unique continuous function on $\ol\Sigma\bs f^{-1}(H)$
that is harmonic in $\Sigma\bs f^{-1}(H)$, vanishes on the boundary $b\Sigma$,
and has a logarithmic pole with multiplicity $m_j$ at each of the points 
$p_j\in f^{-1}(H)$. For Green functions see any of the 
standard sources, or the recent book by Varolin \cite[p.\ 119]{Varolin}.) 
The above inequality follows by observing that
$V_\Omega\circ f + G_{\Sigma,f^*H}$ is a subharmonic function on 
$\Sigma \setminus f^{-1}(H)$ which equals zero on $b\Sigma$ and is locally bounded
from above at each point $p_j\in f^{-1}(H)$; hence it extends as a subharmonic 
function on $\Sigma$, and the maximum principle implies that it is
$\le 0$ on $\Sigma$. 

We now restrict attention to the case when $\Sigma=\D$ is the disc.
The Green function with pole at the point $a\in \D$ equals 
\[
	G_a(\zeta)= \log \left|\frac{\zeta-a}{1-\bar a \zeta}\right|.
\]
For $a\ne 0$ we get $G_a(0)= \log|a|$.
Given a disc $f\in\cA_{\P^n}$ with $f(\bT) \subset \Omega$
and $f^*H = \sum_{j=1}^k m_j a_j$, where $a_j \in \D$ and $m_j\in\N$,
we thus have 
\[
	V_\Omega (f(0)) \le -\sum_{j=1}^k m_j\log|a_j| = J(f). 
\]
Therefore we have for each point $z\in \C^n$ the estimate 
\[
	V_\Omega(z)\le \Psi(z):= \inf_{f\in \cB,\ f(0)=z} J(f).
\]
If the family $\cB$ is good in the sense of Def.\ \ref{def:good}, 
then the function $\Psi\colon \C^n\to\R_+$ is upper
semicontinuous, $\Psi|_\Omega=0$, and $\Psi$ has at most logarithmic
growth at infinity. Poletsky's theorem (Theorem  \ref{Poletsky-Rosay}
for $X=\C^n$) implies that the Poisson envelope $EP_\Psi$
is the maximal plurisubharmonic minorant of $\Psi$. 
Since $\Psi$ grows logarithmically, we infer that 
$EP_\Psi$ belongs to the Lelong class $\cL$. 
Finally, as $V_\Omega\le \Psi$, we have $V_\Omega\le EP_\Psi$,
and hence $V_\Omega = EP_\Psi$ by maximality of $V_\Omega$. 
This proves Theorem \ref{LSThm2}.
\qed
\smallskip

The above proof generalizes immediately to the following situation.
Let $\Omega$ be a non\-empty open subset in an affine algebraic
variety $X\subset \C^n$. 
The Lelong class $\cL_X$, and the maximal function 
$V_{\Omega,X}$, are defined in essentially 
the same way as in the case $X=\C^n$. 
Assume now that $X$ is irreducible,
and let $k=\dim_\C X \in\{1,2,\ldots,n-1\}$. 
Denote by $\ol X$ the closure of $X$ in $\P^n$ 
(an algebraic subvariety of $\P^n$).
Complex curves $C\subset \ol X$ whose boundaries $bC$ are smooth
and contained in the open subset $\Omega \subset X$
fill the entire projective variety $\ol X$.
(We consider only curves that have no isolated points 
in their boundaries.) 
An explicit way to obtain such curves is to take the intersection
of $\ol X$ with a generic projective linear subspaces
$\Lambda\subset \P^n$ of dimension $n-k+1$ 
such that $\Lambda\cap \Omega\ne \emptyset$, 
and then remove from the closed projective curve 
$\Lambda \cap \ol X$ finitely many smoothly bounded 
disc lying in $\Omega$. If $\ol X$ is smooth (without singularities), 
then a generic such intersection will be a smooth embedded 
complex curve with boundary, but in general 
we can not expect it to be a disc. In fact, the degree 
of a generic curve $\Lambda \cap \ol X$ equals the degree of $X$.

Given a curve $C\subset \ol X$ with smooth boundary $bC\subset\Omega$, 
let $f \colon \ol\Sigma \to \ol C$ be a normalization 
of $C$ by a finite bordered Riemann surface, $\Sigma$,
with smooth boundary. 
The normalization map extends smoothly to the boundary. 
Let $f^*H=\sum_{j=1}^d m_j p_j$ denote the intersection divisor of 
the map $g$ with the hyperplane at infinity $H=\P^n\bs \C^n$. 
As before, let $G_{\Sigma,f^*H}$ be the Green function
on $\Sigma$ with logarithmic poles determined by the divisor $f^*H$.
Choose a family $\cB$ of such normalized curves 
$(\Sigma,f,C)$ in $X$ whose images $C=f(\Sigma)$ fill $\ol X$, 
and define the function $\Psi_\cB \colon X\to \R_+$ by
\begin{equation}
\label{eq:PsiB}
	\Psi_\cB(x) = \inf - G_{\Sigma,f^*H}(\zeta) 
	= -\sup G_{\Sigma,f^*H}(\zeta),\quad x\in X,
\end{equation}
where the infimum is over all $(\Sigma,f,C) \in \cB$ 
and points $\zeta\in \Sigma$ with $f(\zeta)=x$. 
By including in $\cB$ all the constant discs $\ol\D \mapsto z \in \Omega$,
and by assuming the existence of continuous 
local families of curves in $\cB$ (in analogy to property (iii) in
Def.\ \ref{def:good}), we insure as before that $\Psi_\cB$ 
is an upper semicontinuous function on $X$ that equals zero
on $\Omega$ and grows at most logarithmically at infinity.
(The continuity of a family of curves with respect to some parameter 
$t$ can be made precise by fixing a smooth oriented 
real surface with boundary, $\ol\Sigma$, and choosing a 
continuous family of almost complex structures $J_t$ on $\ol\Sigma$, 
and a continuous family of holomorphic maps 
$f_t\colon (\ol\Sigma,J_t) \to X$.) 

The argument in the proof of Theorem \ref{LSThm2} shows that 
for each complex curve $f\colon \overline \Sigma \to \ol C \subset \ol X$ 
with $f(b\Sigma) \subset \Omega$ we have
$V_{\Omega,X}\circ f \le - G_{\Sigma,f^*H}$, and hence
$V_{X,\Omega}\le \Psi_\cB$. By applying the general case of 
Theorem \ref{Poletsky-Rosay} we thus obtain  the following 
expression for the maximal function $V_{\Omega,X}$. 

%
%
\begin{thm}
\label{Maximalfunction-variety}
Let $X$ be an irreducible and locally irreducible
algebraic subvariety of $\C^n$, and let $\Omega$ be a nonempty
open set in $X$. Assume that $\cB$ is a good family of complex curves 
in $X$ with boundaries in $\Omega$, and let $\Psi_\cB \colon X\to \R_+$ 
denote the associated function (\ref{eq:PsiB}). Then the Siciak-Zaharyuta 
function $V_{\Omega,X}$ is the envelope of the Poisson
functional $P_\Psi$:
\[
	 V_{\Omega,X}(x) = 
	 \inf \Big\{\int^{2\pi}_0 \Psi_\cB (f(\E^{\I t}))\, 
   \frac{\d t}{2\pi} \colon \ f\in \cO(\clD,X,x) \Big\},
    \quad x\in X. 
\]
\end{thm}

\begin{exa}
An explicit example of a good family $\cB$ that one can use in 
this theorem is obtained by taking all constant discs in $\Omega$,
together with all transverse intersections $L=\Lambda\cap \ol X$
that intersect $\Omega$, where $\Lambda\cong \P^{n-k+1}$ is a projective 
linear subspace of $\P^n$ of dimension $n-\dim X+1$, and removing 
from each such closed curve $L$ a finite family of pairwise 
disjoint, closed, smoothly bounded discs 
$\ol\Delta_1,\ldots,\ol\Delta_m \subset \Omega$  
whose boundaries $b\Delta_j$ belong to the regular locus 
$L_{\reg}$ of $L$. 
The difference $C=L\setminus \bigcup_{j=1}^m \ol\Delta_j$
is then a complex curve in $X$ with smooth boundary 
$bC=\bigcup_{j=1}^m b\Delta_j$ contained in $\Omega$.
\qed\end{exa}

\begin{prob}
Assume that $\Omega\subset X$ are as in Theorem \ref{Maximalfunction-variety}.
Let $\cB$ consist of all complex curves 
$f\colon \ol\Sigma\to \ol C\subset  X$, where $\ol\Sigma$ is a finite 
bordered Riemann surface and $f$ is a holomorphic map 
such that $f(b\Sigma)\subset \Omega$. 
Is the maximal function $V_{\Omega,X}$ then given by 
the Lempert type formula
\[	V_{\Omega,X}(x) = \inf\left\{ - G_{\Sigma,f^*H}(\zeta)\colon 
	  (\ol\Sigma,f) \in\cB,\ f(\zeta)=x \in X\right\} \ ?
\]
\end{prob}

\subsection*{Acknowledgements}
The authors wish to thank Finnur L\'arusson for his suggestion to consider also 
other disc functionals by the methods developed in \cite{DF2} in the context 
of the Poisson functional, and for his proposal to include Lemma \ref{lifting} 
concerning the lifting of discs into a desingularization.
Research on this paper was supported in part by 
grants P1-0291 and J1-2152, Republic of Slovenia.

\end{document}